\input louC.sty

\documentclass[a4paper,draft]{amsproc}
\usepackage{amsrefs}
\usepackage{amsmath}
\usepackage{mathrsfs}
\usepackage{amssymb}
\usepackage{amscd} 

\theoremstyle{plain}
 \newtheorem{thm}{\textbf{Theorem}}[section]

\theoremstyle{definition}
 \newtheorem{exm}{\textbf{Example}}[section]
 \newtheorem{dfn}{\textbf{Definition}}[section]
\theoremstyle{remark}
 \newtheorem{rem}{\textbf{Remark}}[section]
 \numberwithin{equation}{section}

\newtheorem*{theorem*}{\bf{Theorem \ref{M}}}

\renewcommand{\leq}{\leqslant}
\renewcommand{\geq}{\geqslant}

\title[De Branges-Rovnyak spaces with complete Nevanlinna-Pick property]{Which de Branges-Rovnyak spaces have complete Nevanlinna-Pick property?}


\author[Chu]{\bfseries Cheng Chu}

\address{
Department of Mathematics and Statistics\\ 
Laval University  \\ 
Quebec City, QC \\
Canada}
\email{chengchu813@gmail.com}

\begin{document}

\vspace{18mm}
\setcounter{page}{1}
\thispagestyle{empty}

\begin{abstract}
We characterize the de Branges-Rovnyak spaces with complete Nevanlinna-Pick property. Our method relies on the general theory of reproducing kernel Hilbert spaces.
\end{abstract}
\maketitle

\section{Introduction}

Let $X\subset \CC^d$. We say a function $K: X\times X\to\CC$ is a positive kernel on $X$ (written as $K	\succeq 0$) if it is self-adjoint ($K(x,y)=\ol{K(y,x)}$), and for all finite sets $\{\Gl_1,\Gl_2,\dots, \Gl_{n}\}\subset X$, the matrix $(K(\Gl_i,\Gl_j))_{i,j=1}^n$ is positive semi-definite, i.e., for all complex numbers $\Ga_1, \Ga_2, \dots, \Ga_n$,
$$
\sum_{i,j=1}^n \Ga_i\bar{\Ga_j}K(\Gl_i, \Gl_j)\geq 0.
$$

A reproducing kernel Hilbert space $\mathcal{H}$ on $X$ is a Hilbert space of complex valued functions on $X$ such that every point evaluation is a continuous linear functional. Thus for every $w\in X$, there exists an element $K_w\in\mathcal{H}$ such that for each $f\in\mathcal{H}$, $$\langle f, K_w\rangle_{\mathcal{H}} =f(w).$$
Since $K_w(z)=\langle K_w, K_z\rangle_{\mathcal{H}}$, $K$ can be regarded as a function on $X\times X$ and we write $K(z,w)= K_w(z)$. Such $K$ is a positive kernel and the Hilbert space $\mathcal{H}$ with reproducing kernel $K$ is denoted by $\mathcal{H}(K)$. There is a one-to-one correspondence between reproducing kernel Hilbert spaces and positive kernels (see e.g. \cite{ampi}*{Theorem 2.23}).

\begin{dfn}
Let $K: X\times X\to \CC$ be a positive kernel and $K(x ,y)$ never vanishes. We call $K$ a complete Nevanlinna-Pick kernel if for some $\Ga\in X$,
\beq\label{CNP}
1-\frac{K(z,\Ga)K(\Ga, w)}{K(\Ga, \Ga)K(x, y)}\succeq 0.
\eeq
\end{dfn}
In fact, if \eqref{CNP} holds for some $\Ga$, then it must hold for all $\Ga\in X$ \cite{mccult00}*{Theorem 3.1}. We say $\cH(K)$ has complete Nevanlinna-Pick property if $K$ is a complete Nevanlinna-Pick kernel.

Complete Nevanlinna-Pick kernels are related to Nevanlinna-Pick interpolation problem. The problem is to determine, given a finite set $\{x_1, . . . , x_n\}\subset X$, and complex numbers $\Gl_1,...,\Gl_n$, whether there exists a multiplier $f$ in $$Mult(\cH(K))=\{\Gvp: X\to\CC| f\in \cH(K) \Rightarrow  \Gvp f\in \cH(K)\}$$ of multiplier norm at most one that interpolates the data, i.e. satisfies $f(x_i) = \Gl_i$, for $i = 1,...,n$. A necessary condition to solve the Nevanlinna-Pick problem is that the $n\times n$ matrix
$$((1-\bar{\Gl_i}\Gl_j)K(x_i, x_j))_{i, j}$$
is positive semi-definite. A kernel $K$ is called a Pick kernel (or has pick property) if this necessary condition is also sufficient. Similarly one can do matrix interpolation. Given $s\times t$ matrices $W_1,...,W_n$, a necessary condition to solve the $s\times t$ Nevanlinna-Pick problem: $f(x_i)=W_i$, $i=1,...,n$ is that the matrix
$$
((I-W_i^*W_j)K(x_i, x_j))_{i, j}
$$
is positive semi-definite.
If this condition is also sufficient, then we say $K$ has $s\times t$-Pick property. A kernel is a complete Nevanlinna-Pick kernel if and only if it has $s\times t$-Pick property for all $s, t$ \cite{agmc_cnp}.

An important example of a complete Nevanlinna-Pick kernel is the Szeg\H{o} kernel
$$
k^S(z,w)=\frac{1}{1-\bw z}.
$$
It is the reproducing kernel of $H^2$, the Hardy space on the unit disk. Many properties of the Hardy space carry over to other spaces with complete Nevanlinna-Pick property, see \cite{ahmr17}, \cite{ahmr18}, \cite{ahmr19} for exmaples.

A natural question is to decide which reproducing kernel Hilbert spaces have complete Nevanlinna-Pick property. The following are known examples of Hilbert spaces with complete Nevanlinna-Pick property. Let $\DD$ and $\TT$ be the open unit disk and unit circle, respectively, in the complex plain.
\begin{enumerate}
\item Drury-Arverson spaces: Let $B_d$ be the open unit ball in $\CC^d$. The Drury-Arverson space $H^2_d$ is the space of analytic functions on $B_d$ with reproducing kernel $$k_d(z, w)=\frac{1}{1-\la z, w\ra_{\CC^d}}.$$
\item Weighted Hardy spaces \cite{qui93}: Let $w=(w_1,...,w_n,..)$ be a sequence of complex numbers such that $w_n^2\geq w_{n-1}w_{n+1}$, for all $n\geq 2$. Define a space of holomorphic functions $H_w^2$ on $\DD$ by the norm  $$||f||^2_{H^2_w}=\sum_{n=0}^\infty |\hat{f}(n)|^2w_n.$$
    In particular it contains the Hardy space $H^2$ and the Dirichlet space $\cD$.
\item Weighted Dirichlet spaces \cite{shi02}: Let $\mu$ be a finite positive measure on $\ol{\DD}$.
    Let $$U_{\mu}(w)=\int_{\TT}\frac{1-|w|^2}{|w-z|}d\mu(z)+\int_\DD \log \Big|\frac{1-\bw z}{w-z}\Big|^2 \frac{d\mu(z)}{1-|z|^2}.$$
    The weighted Dirichlet space $\cD(\mu)$ \cite{ric91}, \cite{ale93} is defined to be those holomorphic functions on $\DD$ with norm
    $$||f||^2_{\cD(\mu)}=||f||^2_{H^2}+\int_\DD |f^\prime(w)|^2U_{\mu}(w)dA(w).$$
\item Weighted Sobolev spaces \cite{qui93}: Let $w_o$ and $w_1$ be real, strictly positive, continuous functions on $[0, 1]$ and let $w_1\in C^1$. The weighted Sobolev space $W_1^2(w_0, w_1)$ is the space of absolutely continuous functions on $[0,1]$ with finite norm
    $$||f||^2_{W_1^2(w_0, w_1)}=\int_0^1|f(x)|^2w_0(x)dx+\int_0^1 |f^\prime(x)|^2w_1(x)dx.$$
\item Weighted Besov spaces: For $s\in \RR$, the weighted Besov space on the unit ball $\mathbb{B}_d$ of $\CC^d$ is defined by
    $$B_w^s=\{f\in Hol(\mathbb{B}_d): \int_{\mathbb{B}_d} |R^sf|^2 w dV< \infty   \},$$ where $R=\sum_{i=1}^d z^i\frac{\partial}{\partial z_i}$.
    If a weight $w$ satisfies that for some $\Ga > -1$ the ratio $w(z)/(1-|z|^2)^\Ga$ is nondecreasing for $0\leq r_0 < |z| < 1$, then $B_w^s$ has complete Nevanlinna-Pick property whenever $s\geq (\Ga + d)/2$ \cite{ahmr19b}.
\end{enumerate}

In this paper, we show another class of spaces having complete Nevanlinna-Pick property. Those spaces are called de Branges-Rovnyak spaces, introduced by de Branges and Rovnyak in \cite{deb-rov1}. The initial motivation of introducing de Branges-Rovnyak spaces was to provide canonical model spaces for certain types of contractions on Hilbert spaces, which played a fundamental role in de Branges's proof of the famous Bieberbach conjecture \cite{deb3}: If $\disp f(z) = 1+\sum_{n=2}^\infty a_nz^n$ is a univalent analytic function on $\DD$, then $|a_n|\leq n$, for every $n\geq 2$. Subsequently it was realized that these spaces have numerous connections with other topics in complex analysis and operator theory, in particular through Toeplitz operators \cite{sar94}. A recent two-volume monograph \cite{fri16-1}, \cite{fri16-2} presents most of the main developments in this area.
The main result is the following characterization of the de Branges-Rovnyak spaces with complete Nevanlinna-Pick property.
\begin{thm}\label{M}
Let $b$ be a nonconstant function in $H^\infty_1$. Then the de Branges-Rovnyak space $\cH(b)$ has complete Nevanlinna-Pick property if and only if there exists a holomorphic function $h$ such that
\begin{enumerate}
\item $h(b(z))=z$, for all $z\in \DD$.
\item The function $$\frac{z-b(0)}{h(z)}$$ extends to a holomorphic function on $\DD$ with
$$
\Big|\frac{z-b(0)}{h(z)}\Big|\leq |1-\ol{b(0)}z|\q\m{for all}\,z\in \DD.
$$
\end{enumerate}
\end{thm}

In Section 2, we introduce the definition of de Branges-Rovnyak spaces along with basic properties. In Section 3, we present the theory of general reproducing kernels, which played an center role in the proof of the main result. We prove Theorem \ref{M} in Section 4 and discuss some examples in Section 5.

\section{Preliminaries on de Brange-Rovnyak spaces}

Let $A$ be a bounded operator on a Hilbert space $H$. We define the range space $AH$, and endow it with the inner product
\beq\label{ran}\langle Af, Ag \rangle_{AH}=\langle f, g \rangle_{H},\qq f,g\in H \ominus \m{Ker}A.\eeq
The space $AH$ has a Hilbert space structure that makes $A$ a partial isometry on $H$. In particular, if $A$ is injective, then $$||Af||_{AH}=||f||_H.$$

We can view $H^2$ as a closed subspace of $L^2(\TT)$ of functions whose negative Fourier coefficients vanish.
The Toeplitz operator on the Hardy space $H^2$ with symbol $f$ in $L^\infty(\TT)$ is defined by
$$T_f (h) = P(fh),$$ for $h\in H^2$. Here $P$ is the orthogonal projection from $L^2(\TT)$ to $H^2$.
\begin{dfn}
Let $b$ be a function in $H^\infty_1$, the closed unit ball of $H^\infty$. The de Branges-Rovnyak space $\cH(b)$ is defined to be the range space
$$\cH(b):=(I-T_b T_{\bar b})^{1/2} H^2.$$
\end{dfn}

There are two special cases for $\cH(b)$ spaces. If $||b||_\infty<1$, then $\cH(b)$ is just a re-normed version of $H^2$. If $b$ is an inner function, then $$\cH(b)=H^2\ominus bH^2$$ is a closed subspace of $H^2$, the so-called model space (see \cite{garros} for a brief survey).

The theory of $\cH(b)$ spaces is pervaded by a fundamental dichotomy, when $b$ is an extreme point of $H^\infty_1$ and when it is not. The nonextreme case includes $||b||_\infty<1$ and the extreme case includes $b$ is an inner function.
The following theorem is shown in \cite{sar94}*{Chapter IV, V}.
\begin{thm}
Let $b\in H^\infty_1$. The following are equivalent:
\begin{enumerate}
\item $\cH(b)$ contains all the polynomials.
\item $b$ is not an extreme point of the unit ball of $H^\infty_1$.
\item $\log(1-|b|^2)\in L^1(\TT)$.
\end{enumerate}
\end{thm}

If $b$ is a constant, then $\cH(b)$ is a rescaling of $H^2$ and has complete Nevanlinna-Pick property. Some nontrivial examples are known.
The following theorem shows that certain nonextreme $\cH(b)$ spaces are equal to a weighted Dirichlet space, which has complete Nevanlinna-Pick property.
\begin{thm}\cite{cgr}*{Theorem 3.1}
Let
$$
b(z)=\frac{\sqrt{\tau}\Ga\bar{\Gl} z}{1-\tau\bar{\Gl} z},
$$
where $\Gl\in \TT, \Ga\in\CC, \tau\in (0,1]$ and $\tau+1/\tau=2+|\Ga|^2.$
Then $\cH(b)$ coincides with a weighted Dirichlet space $\cD(\mu)$
\end{thm}
We will show in Section 5 that extreme $\cH(b)$ spaces may also have complete Nevanlinna-Pick property.

\section{Reproducing kernel Hilbert spaces}\label{Pre}
In this section, we present some basic theory of reproducing kernel Hilbert spaces. For more information about reproducing kernels and their associated Hilbert spaces, see \cite{aro50}, \cite{pau16}.

The space $\cH(b)$ has reproducing kernel \cite{sar94}*{II-3}
\beq\label{kb}
k^b(z,w)=\frac{1-\ol{b(w)}b(z)}{1-\bw z}.
\eeq
In general, let $A$ be a bounded linear operator on $\cH(K)$. The range space $A \cH(K)$, defined in \eqref{ran}, is also a reproducing kernel Hilbert space with kernel function $K^{A\cH(K)}_w=AA^*K_w$.

For two positive kernels $K_1, K_2$, we write $$K_1	\preceq K_2$$ to mean that $$\label{n} K_2-K_1\succeq 0.$$
\begin{thm}\label{in}\cite{pau16}*{Theorem 5.1}
Let $\mathcal{H}(K_1)$ and $\mathcal{H}(K_2)$ be reproducing kernel Hilbert spaces on $X$. Then $$\mathcal{H}(K_1) \subset \mathcal{H}(K_2)$$ if and only if there is some constant $\Gd>0$ such that
$$
K_1\preceq \Gd^2 K_2.
$$
Moreover, $||f||_{\cH(K_2)}\leq \Gd||f||_{\cH(K_1)}$, for all $f\in \cH(K_1)$.
\end{thm}

It is easy to check the sum of two positive kernels is still a positive kernel, and the corresponding reproducing kernel Hilbert space is described below.
\begin{thm}\label{+}\cite{pau16}*{Section 5.2}
Let $\mathcal{H}(K_1)$ and $\mathcal{H}(K_2)$ be reproducing kernel Hilbert spaces on $X$ and let $K=K_1+K_2$. Then
$$\cH(K)=\{f=f_1+f_2: f_1\in \cH(K_1), f_2\in \cH(K_2)\}.
$$
For every $f\in \cH(K)$,
$$
||f||^2_{\cH(K)}=\m{min}\{||f_1||^2_{\cH(K_1)}+||f_2||^2_{\cH(K_2)}: f=f_1+f_2: f_1\in \cH(K_1), f_2\in \cH(K_2)  \}.
$$
Moreover, if $\cH(K_1)\cap\cH(K_2)=\{0\}$, then the sum is direct, i.e.
$$
||f_1+f_2||^2_{\cH(K)}=||f_1||^2_{\cH(K_1)}+||f_2||^2_{\cH(K_2)}.
$$
\end{thm}

Let $\Gvp: X\to S$ be a function. For a positive kernel $K: X\times X\to\CC$, we write $K\circ\Gvp$ to be the function given by $(K\circ\Gvp)(z,w)=K(\Gvp(z),\Gvp(w))$. Then $K\circ\Gvp$ is also a positive kernel and we have the following pull-back theorem.
\begin{thm}\label{comp}\cite{pau16}*{Theorem 5.7}
Let $\Gvp: S\to X$ be a function and let $K: X\times X\to\CC$ be a positive kernel. Then
$$
\cH(K\circ\Gvp)=\{f\circ \Gvp: f\in \cH(K)\}.
$$
And for $u\in \cH(K\circ\Gvp)$,
$$
||u||_{\cH(K\circ\Gvp)}=\m{min} \{||f||_{\cH(K)}: u=f\circ\Gvp   \}.
$$
\end{thm}
We shall also use $\Gvp^*(\cH(K))$ to denote $\cH(K\circ\Gvp)$.

\begin{rem}\label{A}
If $\cH(K)$ is a space of analytic functions and $\Gvp$ is an analytic map, then we simply have
$$
||f\circ\Gvp||_{\cH(K\circ\Gvp)}=||f||_{\cH(k)}.
$$
\end{rem}

\section{Proof of Theorem \ref{M}}
For convenience, we restate Theorem \ref{M} here.
\begin{theorem*}
Let $b$ be a nonconstant function in $H^\infty_1$. Then the de Branges-Rovnyak space $\cH(b)$ has complete Nevanlinna-Pick property if and only if there exists a holomorphic function $h$ such that
\begin{enumerate}
\item $h(b(z))=z$, for all $z\in \DD$.
\item The function $$\frac{z-b(0)}{h(z)}$$ extends to a holomorphic function on $\DD$ with
$$
\Big|\frac{z-b(0)}{h(z)}\Big|\leq |1-\ol{b(0)}z|\q\m{for all}\,z\in \DD.
$$
\end{enumerate}
\end{theorem*}

\begin{proof}
Let $k^b(z,w)$ be the reproducing kernel of $\cH(b)$ given in \eqref{kb}.
By definition, $\cH(b)$ has complete Nevanlinna-Pick property if and only if
\beq\label{m0}
1-\frac{k^b(z,0)k^b(0,w)}{k^b(0,0)k^b(z,w)}\succeq 0.
\eeq
Let $b(0)=a\, (|a|<1)$ and let $$F(z)=k^b(z,0)=1-\ba b(z).$$ Then \eqref{m0} simplifies to
$$
F(0)-\frac{\ol{F(w)}F(z)(1-\bw z)}{1-\ol{b(w)}b(z)}\succeq 0.
$$
Expanding and rearranging, we see that it is equivalent to
\beq\label{m1}
\frac{\ol{F(w)}{F(z)}}{1-\ol{b(w)}b(z)}\preceq \frac{\ol{wF(w)}zF(z)}{1-\ol{b(w)}b(z)}+F(0).
\eeq
Note that both sides of \eqref{m1} are positive kernels. Let
$$K_1(z,w)=\frac{\ol{F(w)}{F(z)}}{1-\ol{b(w)}b(z)}=\frac{(\ol{1-\ba b(w)}){(1-\ba b(z))}}{1-\ol{b(w)}b(z)},$$
$$K_2(z,w)=\frac{\ol{wF(w)}zF(z)}{1-\ol{b(w)}b(z)}=\bw z\frac{(\ol{1-\ba b(w)}){(1-\ba b(z))}}{1-\ol{b(w)}b(z)},$$
and $$K_0(z,w)=F(0)=1-|a|^2.$$
By Theorem \ref{in} and \eqref{m1}, $\cH(b)$ has complete Nevanlinna-Pick property if and only if $\cH(K_1)$ is contractively contained in $\cH(K_2+K_0)$.

Next we identify the corresponding reproducing kernel Hilbert spaces.
Clearly, $\cH(K_0)$ consists of constant functions $f_c=c$, and $$||f_c||_{\cH(K_0)}=\frac{|c|}{\sqrt{1-|a|^2}}.$$
Notice that the range space $T_{1-\ba z}H^2$ has reproducing kernel
$$T_{1-\ba z}T_{1-\ba z}^*k_w^S=\frac{(\ol{1-\ba w}){(1-\ba z)}}{1-\bw z}.$$
By Theorem \ref{comp}, $$\cH(K_1)=b^*(T_{1-\ba z}H^2)=\{(1-\ba b(z))f(b(z)): f\in H^2\}.$$
For $|a|<1$, $T_{1-\ba z}$ is injective. Thus by Remark \ref{A}, we have
$$
||(1-\ba b(z))f(b(z))||_{\cH(K_1)}=||(1-\ba z)f(z)||_{T_{1-\ba z}H^2}=||f||_{H^2}.
$$
On the other hand, $\cH(K_2)$ is the range space
$$T_z\cH(K_1)=\{z(1-\ba b(z))g(b(z)): g\in H^2    \},$$
with norm (here $T_z$ is also injective)
$$
||z(1-\ba b(z))g(b(z))||_{\cH(K_2)}=||(1-\ba b(z))g(b(z))||_{\cH(K_1)}=||g||_{H^2}.
$$
It is also clear that $\cH(K_2)\cap \cH(K_0)=\{0\}$, thus the sum on the righthand side of \eqref{m1} is direct.
Therefore, by Theorem \ref{in}, \ref{+}, equation \eqref{m1} is equivalent to the following statement:
For every $f\in H^2$, there exists $g\in H^2$ such that
\beq\label{m2}
(1-\ba b(z))f(b(z))=z(1-\ba b(z))g(b(z))+f(a)(1-|a|^2),
\eeq
and
\beq\label{m3}
||g||_{H^2}^2+(1-|a|^2)|f(a)|^2\leq ||f||_{H^2}^2.
\eeq
In particular, taking $f=1$, there exists $g_0\in H^2$ such that
\beq\label{m4}
1-\ba b(z)=z(1-\ba b(z))g_0(b(z))+(1-|a|^2).
\eeq

We claim that $b$ has to be injective. Indeed, if $b(z_1)=b(z_2)$ for some pair $z_1\neq z_2$, then the condition \eqref{m4} forces $g_0$ to vanish at $b(z_1)$.
But if $b(z_1) = b(z_2)$, then for each $\tiz_1$ near to $z_1$, there exists $\tiz_2$ near to $z_2$ such that $b(\tiz_1)= b(\tiz_2)$, and so $g_0$ has to vanish at all points near to $b(z_1)$. This forces $g_0\equiv 0$, which in turn implies $b$ is a constant.

Since $b$ is injective, define $h=b^{-1}: b(\DD)\to\DD$, which means $h(b(z))=z$ for all $z\in \DD$. Then \eqref{m2} implies
\beq\label{m5}
g(z)=\frac{1}{h(z)}\Big(f(z)-\frac{f(a)(1-|a|^2)}{1-\ba z} \Big),
\eeq
for all $z\in b(\DD)$. Thus the right hand side of \eqref{m5} extends to a holomorphic function on $\DD$ for all $f\in H^2$.
Notice that
$$
f(z)-\frac{f(a)(1-|a|^2)}{1-\ba z}=f(z)-f(a)(1-|a|^2)k^S_a(z)
$$
is precisely the projection of $f$ onto the subspace $$\{k^S_a\}^\perp=\{f\in H^2: f(a)=0\}.$$
Thus we see that the function $$\frac{z-a}{h(z)}$$ extends to a holomorphic function on $\DD$.

Now \eqref{m3} implies
\beq\label{m6}
\Big|\Big|\frac{1}{h(z)}\Big(f(z)-\frac{f(a)(1-|a|^2)}{1-\ba z} \Big)\Big|\Big|_{H^2}^2\leq ||f||_{H^2}^2-(1-|a|^2)|f(a)|^2,
\eeq
for every $f\in H^2$.
Also we have
\begin{align*}
&\Big|\Big|f(z)-\frac{f(a)(1-|a|^2)}{1-\ba z}\Big|\Big|_{H^2}^2\\
=&||f||_{H^2}^2+||f(a)(1-|a|^2)k^S_a ||_{H^2}^2-2\,\m{Re}\,\la f, f(a)(1-|a|^2)k^S_a \ra_{H^2}\\
=&||f||_{H^2}^2+|f(a)|^2(1-|a|^2)-2 \ol{f(a)}(1-|a|^2)f(a)\\
=&||f||_{H^2}^2-(1-|a|^2)|f(a)|^2,
\end{align*}
which is exactly the righthand side of \eqref{m6}.
Therefore, \eqref{m6} shows that
\beq\label{m7}
\Big|\Big|\frac{u}{h(z)}\Big|\Big|_{H^2}\leq ||u||_{H^2},
\eeq
for every $u\in \{k^S_a\}^\perp$. Since the functions in $\{k^S_a\}^\perp$ are those of the form
$$
\frac{z-a}{1-\ba z}\cdot u_1(z),\q u_1\in H^2,
$$
and $$\Big|\frac{z-a}{1-\ba z}\cdot u_1(z)\Big|=|u_1(z)|\q a.e. \,\,\m{on}\,\,\TT,$$
we see that
$$
\frac{z-a}{h(z)}\cdot\frac{1}{1-\ba z}
$$
is a contractive multiplier of $H^2$. Thus
$$
\Big|\Big|\frac{z-a }{h(z)}\cdot\frac{1}{1-\ba z}\Big|\Big|_\infty\leq 1,
$$
which gives
$$
\Big|\frac{z-b(0)}{h(z)}\Big|\leq |1-\ol{b(0)}z|,\q \forall z\in \DD.
$$
This proves the ``only if'' part of the theorem, and the ``if'' part follows just by working backwards.

\end{proof}

\section{Examples}
Theorem \ref{M} shows that, a necessary condition for $\cH(b)$ to have complete Nevanlinna-Pick property is that $b$ be injective on $\DD$. However, the following example shows this is not sufficient.
\begin{exm}
Let $P$ be a polygon whose closure is a subset of $\DD$ and which has a vertex $v$ at which two sides make an angle that is an irrational multiple of $\pi$. Let $b$ be a conformal map of $\DD$ onto $P$ with $b(0) = 0$. Then obviously $b$ is injective. However, there can exist no function $h$ satisfying condition (2) in Theorem \ref{M} because $h$ would have to be holomorphic at $v$, and the irrationality condition on the angle of $P$ at $v$ renders this impossible.
\end{exm}

We have some explicit examples of nonextreme $\cH(b)$ spaces with complete Nevanlinna-Pick property.
\begin{exm}
Let $b\in H^\infty_1$. If $b$ has the form $$\frac{z+A}{B}, (B\neq 0)\q \m{or}\q\frac{Az}{z+B}, (|A|\geq 1, B\neq 0),$$ where $A, B$ are constants, then $\cH(b)$ has complete Nevanlinna-Pick property.
\end{exm}
\begin{proof} We verify the conditions given by Theorem \ref{M} in both cases.
\begin{enumerate}
\item Let $$b(z)=\frac{z+A}{B}.$$ Then $b(0)=\frac{A}{B}$, and $||b||_\infty\leq 1$ implies that $|A|+1\leq |B|$. Taking $h(z)=Bz-A$, then $h(b(z))=z$. For $z\in \DD$, $$\Big|\frac{z-b(0)}{1-\ol{b(0)}z}\Big|=\frac{|Bz-A|}{|B\bz-A|}\leq \frac{|Bz-A|}{|B|-|A|}\leq |h(z)|.$$
\item Let $$b(z)=\frac{Az}{z+B}.$$ Then $b(0)=0$, and $||b||_\infty\leq 1$ implies that $|A|\leq |B|-1$. Taking $h(z)=\frac{Bz}{A-z}$, then $h$ is holomorphic on $\DD$ with $h(b(z))=z$. For $z\in \DD$, $$|h(z)|=\frac{|B|}{|A-z|}|z|\geq \frac{|B|}{|A|+1}|z|\geq |z|.$$
\end{enumerate}
\end{proof}

Suppose $b$ is an inner function and the model space $\cH(b)$ has complete Nevanlinna -Pick property. Then Theorem \ref{M} shows that $b$ is injective, and thus it has to be a disk automorphism. We can take $h=b^{-1}$ and it satisfies \eqref{cM}. Therefore we proved
\begin{thm}
The only model spaces with complete Nevanlinna-Pick property are one dimensional model spaces.
\end{thm}

However, we can show that there are lots of de Branges-Rovnyak spaces (including extreme cases) with complete Nevanlinna-Pick property.
\begin{thm}\label{c1}
Let $\GO$ be a simply connected domain such that $\DD\subset\GO\neq\CC$. Let $\phi: \GO\to \DD$ be the Riemann map, and let $b=\phi|_{\DD}$. Then $\cH(b)$ has complete Nevanlinna-Pick property.
\end{thm}
\begin{proof}
Define $h: \DD\to \GO$ by $h=\phi^{-1}$. Clearly $h(b(z))=z$ for all $z\in \DD$. Also, by the maximum modulus principle applied to
$$
\frac{\phi(z)-\phi(0)}{z(1-\ol{\phi(0)}\phi(z))},
$$
we have
$$
\Big|\frac{\phi(z)-\phi(0)}{1-\ol{\phi(0)}\phi(z)}\Big|\leq |z|,\q\m{for all}\,z\in\GO,
$$
which translate to
$$
|h(z)|\geq \Big|\frac{z-b(0)}{1-\ol{b(0)}z}\Big|,\q \m{for all}\, z\in \DD.
$$
The result now follows from Theorem \ref{M}.
\end{proof}

\begin{rem}
Taking $\GO$ to be a domain that contains an arc of the unit circle as part of its boundary, the corresponding function $b$ will have (radial limits with) modulus $1$ on an arc, and so it will certainly be an extremal function. This shows that extreme functions can give $\cH(b)$ spaces with complete Nevanlinna-Pick property.
\end{rem}

The condition in Theorem \ref{c1} is sufficient for $\cH(b)$ to have complete Nevanlinna-Pick property, but it is not necessary.
\begin{exm}
Let $D= \{x + iy : |x| < 1, |y| < 3\pi/2 \}$, let $\psi$ be a conformal map of $\DD$ onto $D$ with $\psi(0)=0$, and let $$h(z)= Ce^{\psi(z) - 1},$$ where $C$ is a constant. Note that $h$ is not injective. However, if $C$ is chosen large enough, then $|h(z)| \geq |z|$ for all $z\in \DD$, and it maps a neighborhood $N$ of $0$ conformally onto $\DD$. Thus, if $b = (h|_N )^{-1}$, then by Theorem \ref{M}, $\cH(b)$ has complete Nevanlinna-Pick property. On the other hand, $b$ does not arise in the manner predicated by Theorem \ref{c1} because if $\phi$ existed, then $\phi$ would have to equal $h$, and $h$ is not injective on $D$.
\end{exm}

\section*{Acknowledgments}
I am grateful to Thomas Ransford for valuable discussions and for suggesting Theorem \ref{c1} and some of the examples.

\bibliography{references}

\begin{bibdiv}
\begin{biblist}

\bib{ahmr17}{article}{
      author={A.~Aleman, J.E.~McCarthy, M.~Hartz},
      author={Richter, S.},
       title={{The Smirnov class for spaces with the complete Pick property}},
        date={2017},
     journal={J. Lond. Math. Soc.},
      volume={96},
      number={1},
       pages={228\ndash 242},
}

\bib{ahmr18}{article}{
      author={A.~Aleman, J.E.~McCarthy, M.~Hartz},
      author={Richter, S.},
       title={{Factorizations induced by complete Nevanlinna-Pick factors}},
        date={2018},
     journal={Adv. Math.},
      volume={335},
       pages={372\ndash 404},
}

\bib{ahmr19}{article}{
      author={A.~Aleman, J.E.~McCarthy, M.~Hartz},
      author={Richter, S.},
       title={{Interpolating sequences in spaces with the complete Pick
  property}},
        date={2019},
     journal={Int. Math. Res. Not.},
      number={12},
       pages={3832\ndash 3854},
}

\bib{ahmr19b}{book}{
      author={A.~Aleman, J.E.~McCarthy, M.~Hartz},
      author={Richter, S.},
       title={{Radially weighted Besov spaces and the Pick property}},
      series={{Analysis of Operators on Function Spaces, Trends in
  Mathematics}},
   publisher={{Birkh\"auser}},
     address={Cham},
        date={2019},
}

\bib{agmc_cnp}{article}{
      author={Agler, J.},
      author={M\raise.45ex\hbox{c}Carthy, J.E.},
       title={Complete {Nevanlinna-Pick} kernels},
        date={2000},
     journal={J. Funct. Anal.},
      volume={175},
      number={1},
       pages={111\ndash 124},
}

\bib{ampi}{book}{
      author={Agler, J.},
      author={M\raise.45ex\hbox{c}Carthy, J.E.},
       title={{P}ick interpolation and {H}ilbert function spaces},
   publisher={American Mathematical Society},
     address={Providence},
        date={2002},
}

\bib{ale93}{unpublished}{
      author={Aleman, A.},
       title={The multiplication operators on {Hilbert} spaces of analytic
  functions},
        date={1993},
        note={Habilitationsschrift},
}

\bib{aro50}{article}{
      author={Aronszajn, N.},
       title={Theory of reproducing kernels},
        date={1950},
     journal={Trans. Amer. Math. Soc.},
      volume={68},
       pages={337\ndash 404},
}

\bib{deb3}{article}{
      author={de~Branges, L.},
       title={A proof of the {Bieberbach} conjecture},
        date={1985},
     journal={Acta Math.},
      volume={154},
       pages={137\ndash 152},
}

\bib{deb-rov1}{book}{
      author={de~Branges, L.},
      author={Rovnyak, J.},
       title={Square summable power series},
   publisher={Holt, Rinehart, and Winston},
     address={New York},
        date={1966},
}

\bib{fri16-1}{book}{
      author={Fricain, E.},
      author={Mashreghi, J.},
       title={The theory of {$\mathcal{H}(b)$} spaces},
      series={New Mathematical Monographs, 20},
   publisher={Cambridge University Press},
     address={Cambridge},
        date={2016},
      volume={1},
}

\bib{fri16-2}{book}{
      author={Fricain, E.},
      author={Mashreghi, J.},
       title={The theory of {$\mathcal{H}(b)$} spaces},
      series={New Mathematical Monographs, 21},
   publisher={Cambridge University Press},
     address={Cambridge},
        date={2016},
      volume={2},
}

\bib{garros}{book}{
      author={Garcia, S.},
      author={Ross, W.},
       title={Model spaces: a survey},
      series={Contemp. Math.},
   publisher={Amer. Math. Soc.},
     address={Providence, RI},
        date={2015},
      volume={638},
}

\bib{mccult00}{article}{
      author={M{\raise.45ex\hbox{c}C}ullough, S.A.},
      author={Trent, T.T.},
       title={Invariant subspaces and {Nevanlinna-Pick} kernels},
        date={2000},
     journal={J. Funct. Anal.},
      volume={178},
       pages={226\ndash 249},
}

\bib{cgr}{article}{
      author={N.~Chevrot, D.~Guillot},
      author={Ransford, T.},
       title={{De Branges¨CRovnyak spaces and Dirichlet spaces}},
        date={2010},
     journal={J. Funct. Anal.},
      volume={259},
       pages={2366\ndash 2383},
}

\bib{pau16}{book}{
      author={Paulsen, V.I.},
      author={Raghupathi, M.},
       title={An introduction to the theory of reproducing kernel {Hilbert}
  spaces},
   publisher={Cambridge University Press},
     address={Cambridge},
        date={2016},
}

\bib{qui93}{article}{
      author={Quiggin, P.},
       title={For which reproducing kernel {Hilbert spaces is Pick's theorem
  true?}},
        date={1993},
     journal={Integr. Equ. Oper. Th.},
      volume={16},
      number={2},
       pages={244\ndash 266},
}

\bib{ric91}{article}{
      author={Richter, S.},
       title={A representation theorem for cyclic analytic two-isometries},
        date={1991},
     journal={Trans. Amer. Math. Soc.},
      volume={328},
       pages={325\ndash 349},
}

\bib{sar94}{book}{
      author={Sarason, D.},
       title={{Sub-Hardy Hilbert spaces in the unit disk}},
   publisher={University of Arkansas Lecture Notes, Wiley},
     address={New York},
        date={1994},
}

\bib{shi02}{article}{
      author={Shimorin, S.},
       title={Complete {Nevanlinna-Pick} property of {Dirichlet-type} spaces},
        date={2002},
     journal={J. Funct. Anal.},
      volume={201},
       pages={276\ndash 296},
}

\end{biblist}
\end{bibdiv}
\end{document}